\documentclass[11pt,reqno]{amsart}
\usepackage{amssymb,mathrsfs,graphicx}
\usepackage{ifthen}
\usepackage{colortbl}
\definecolor{black}{rgb}{0.0, 0.0, 0.0}
\definecolor{red}{rgb}{1.0, 0.5, 0.5}
\provideboolean{shownotes} 
\setboolean{shownotes}{true} 
\newcommand{\margnote}[1]{
\ifthenelse{\boolean{shownotes}}%
{\marginpar{\raggedright\tiny\texttt{#1}}}%
{}%
}
\newcommand{\hole}[1]{
\ifthenelse{\boolean{shownotes}}%
{\begin{center} \fbox{ \rule {.25cm}{0cm} \rule[-.1cm]{0cm}{.4cm}
\parbox{.85\textwidth}{\begin{center} \texttt{#1}\end{center}} \rule
{.25cm}{0cm}}\end{center}} {} }

\def\d{\,\mathrm{d}}
\def\tot#1#2{\frac{d #1}{d #2}}

\title[Cucker-Smale type model with distributed time delays]{Emergent behavior of Cucker-Smale model with normalized weights and distributed time delays}

\author[Choi]{Young-Pil Choi}
\address[Young-Pil Choi]{\newline Department of Mathematics and Institute of Applied Mathematics, Inha University,
    \newline Incheon 402-751, Republic of Korea}
\email{ypchoi@inha.ac.kr}

\author[Pignotti]{Cristina Pignotti}
\address[Cristina Pignotti]{\newline
Dipartimento di Ingegneria e Scienze dell'Informazione e Matematica, Universit\`a di L'Aquila,
\newline Via Vetoio, 67010 L'Aquila, Italy}
\email{pignotti@univaq.it}

\numberwithin{equation}{section}

\newtheorem{theorem}{Theorem}[section]
\newtheorem{lemma}{Lemma}[section]

\newtheorem{remark}{Remark}[section]
\newtheorem{definition}{Definition}[section]

\def\({\begin{eqnarray}}
\def\){\end{eqnarray}}
\def\[{\begin{eqnarray*}}
\def\]{\end{eqnarray*}}

\newcommand{\R}{\mathbb R}

\newcommand{\mc}{\mathcal C}

\newcommand{\bq}{\begin{equation}}
\newcommand{\eq}{\end{equation}}

\newcommand{\lt}{\left}
\newcommand{\rt}{\right}

\newcommand{\pa}{\partial}

\def\Lyap{\mathcal{L}}

\def\P{\mathcal{P}}
\def\N{\mathbb{N}}
\def\bx{\mathbf{x}}
\def\bv{\mathbf{v}}

\begin{document}
\allowdisplaybreaks


\subjclass[]{}


\begin{abstract} We study a Cucker-Smale-type flocking model with distributed time delay where individuals interact with each other through normalized communication weights. Based on a Lyapunov functional approach, we provide sufficient conditions for the velocity alignment behavior. We then show that as the number of individuals $N$ tends to infinity, the $N$-particle system can be well approximated by a delayed Vlasov alignment equation. Furthermore, we also establish the global existence of measure-valued solutions for the delayed Vlasov alignment equation and its large-time asymptotic behavior.
\end{abstract}

\maketitle \centerline{\date}


%
%
%
%
\section{Introduction}
In the last years the study of collective behavior of multi-agent systems has attracted the interest of many researchers in different scientific fields, such as biology, physics, control theory, social sciences, economics.
The celebrated Cucker-Smale model has been proposed and analyzed in \cite{CS1, CS2} to describe situations in which different agents, e.g. animals groups, reach a consensus (flocking), namely they align and move as a flock, based on a simple rule: each individual adjusts its velocity taking into account other agents' velocities.

In the original papers a symmetric interaction potential is considered. Then,
the case of non-symmetric interactions
has been studied by  Motsch and Tadmor \cite{Motsch-Tadmor}. Several generalizations and variants have been introduced to cover various applications' fields, e.g. more general interaction  rates and singular potentials \cite{CCH2, CCMP, CuckerDong, HaSlemrod, Mech, Pes}, cone-vision constraints \cite{Vicsek}, presence of leadership \cite{Couzin, Shen}, noise terms \cite{CuckerMordecki, EHS, HaLee},  crowds dynamics \cite{Cristiani, Lemercier}, infinite-dimensional models \cite{Albi, Bellomo, CFRT, Ha-Liu, Tadmor-Ha, Toscani}, control problems \cite{Borzi, Caponigro, CKPP, PRT}. We refer to \cite{CCP, CHL} for recent surveys on the Cucker-Smale type flocking models and its variants.

It is natural to introduce a time delay in the model, as a reaction time or a time to receive environmental information. The presence of a time delay makes the problem more difficult to deal with. Indeed, the time delay destroys some symmetry features of the model which are crucial in the proof of convergence to consensus.
For this reason, in spite of a great amount of literature on Cucker-Smale models, only a few papers are available concerning Cucker-Smale model with time delay \cite{CH17, CH19, CL, EHS, PT}. Cucker-Smale models with delay effects are also studied in \cite{PR, PRV} when a hierarchical structure is present, namely the agents are ordered in a specific order depending on which other agents they are leaders of or led by.

Here we consider a distributed delay term, i.e. we assume that the agent $i,$ $i=1,\dots,N,$ changes its velocity depending on the information received from other agents on a time interval $[t-\tau(t) , t].$ Moreover, we assume normalized communication weights (cf \cite{CH17}).
Let us consider a finite number $N\in\N$ of autonomous individuals
located in $\R^d$, $d\geq 1$. Let $x_i(t)$ and $v_i(t)$ be the position and velocity of $i$th individual. Then, in the current work,  we will deal with a Cucker-Smale model with distributed time delays. Namely,
let $\tau :[0, +\infty)\rightarrow (0, +\infty),$  be the time delay function belonging to $W^{1,\infty}(0,+\infty)$.
Throughout this paper, we assume that the time delay function $\tau$ satisfies
\begin{equation}\label{tau2}
\tau^\prime (t)\le 0\, \quad \mbox{and} \quad \tau(t)\ge \tau_*, \quad \mbox{for}\ t\ge0,
\end{equation}
for some positive constant $\tau_*.$
Then, denoting $\tau_0:=\tau(0),$ we have

\begin{equation}\label{tau1}
\tau_*\le \tau (t)\le \tau_0\, \quad \mbox{for}\ t\ge 0.
\end{equation}
It is clear that the constant time delay $\tau(t) \equiv \bar\tau > 0$ satisfies the conditions above.

Our main system is given by
\begin{align}\label{main_eq}
\begin{aligned}
\tot{x_i(t)}{t} &= v_i(t),\qquad i=1,\cdots,N, \quad t >0,\\
\tot{v_i(t)}{t} & = \frac 1 {h(t)} \sum_{k=1}^N \int_{t-\tau(t)}^t \alpha(t-s)\phi(x_k(s),x_i(t)) (v_k(s) - v_i(t))\,ds,
\end{aligned}
\end{align}
where $\phi(x_k(s),x_i(t))$ are the normalized communication weights given by
\begin{equation}\label{interact}
\phi(x_k(s),x_i(t)) = \left\{ \begin{array}{ll}
\displaystyle \frac{\psi(|x_k(s) - x_i(t)|)}{\sum_{j\neq i} \psi(|x_{j}(s) - x_i(t)|)} & \textrm{if $k \neq i$,}\\[4mm]
 0 & \textrm{if $k=i$,}
  \end{array} \right.
\end{equation}
with the influence function $\psi: [0,\infty) \to (0,\infty).$ Throughout this paper, we assume that the influence function $\psi$ is bounded, positive,
nonincreasing and Lipschitz continuous on $[0,\infty)$, with $\psi(0) = 1$. Moreover, $\alpha: [0,\tau_0] \to [0,\infty)$ is a weight function satisfying
\[
\int_0^{ \tau_*} \alpha(s)\, ds  > 0,
\]
and
$$h(t):=\int_0^{\tau (t)} \alpha (s)\, ds\,, \quad t\ge 0.$$
We consider the system subject to the initial datum
\(\label{IC0}
   x_i(s) =: x^0_i(s), \quad v_i(s) =: v^0_i(s),\qquad i=1,\cdots,N, \quad s \in [-\tau_0,0],
\)
i.e., we prescribe the initial position and velocity trajectories $x^0_i, v^0_i\in\mathcal C([-\tau_0,0]; \R^d)$.

For the particle system \eqref{main_eq}, we will first discuss the asymptotic behavior of solutions in Section \ref{sec_par}. Motivated from \cite{CH17, Ha-Liu, Motsch-Tadmor}, we derive a system of dissipative differential inequalities, see Lemma \ref{lem_sddi}, and construct a Lyapunov functional. This together with using Halanay inequality enables us to show the asymptotic velocity alignment behavior of solutions under suitable conditions on the initial data.

We next derive, analogously to \cite{CH17} where the case of a single pointwise time delay is considered, a delayed Vlasov alignment equation from the particle system \eqref{main_eq} by sending the number of particles $N$ to infinity:
\bq\label{main_eq2}
\pa_t f_t + v \cdot \nabla_x f_t + \nabla_v \cdot \displaystyle \lt(\frac{1}{h(t)}\int_{t - \tau(t)}^t \alpha
(t-s)F[f_s]\,dsf_t\rt) = 0,
\eq
where $f_t=f_t(x,v)$ is the one-particle distribution function on the phase space $\R^d \times \R^d$ and the velocity alignment force $F$ is given by
\[
F[f_s](x,v) :=  \frac{\displaystyle \int_{\R^d \times \R^d} \psi(|x-y|)(w-v)f_s(y,w)\,dydw}{\displaystyle \int_{\R^d \times \R^d} \psi(|x-y|)f_s(y,w)\,dydw}.
\]
We show the global-in-time existence and stability of measure-valued solutions to \eqref{main_eq2} by employing the Monge-Kantorowich-Rubinstein distance. As a consequence of the stability estimate, we discuss a mean-field limit providing a quantitative error estimate between the empirical measure associated to the particle system \eqref{main_eq} and the measure-valued solution to \eqref{main_eq2}. We then extend the estimate of large behavior of solutions for the particle system \eqref{main_eq} to the one for the delayed Vlasov alignment equation \eqref{main_eq2}. For this, we use the fact that the estimate of large-time behavior of solutions to the particle system \eqref{main_eq} is independent of the number of particles.  By combining this and the mean-field limit estimate, we show that the diameter of velocity-support of solutions of \eqref{main_eq2} converges to zero as time goes to infinity. Those results will be proved in Section \ref{MF}.


%
\section{Asymptotic flocking behavior of the particle system \eqref{main_eq}}\label{sec_par}
We start with presenting a notion of flocking behavior for the system \eqref{main_eq}, and for this we introduce the spatial and, respectively, velocity diameters as follows:
\(  \label{dXdV}
   d_X(t) := \max_{1 \leq i,j \leq N}|x_i(t) - x_j(t)| \quad \mbox{and} \quad d_V(t) := \max_{1 \leq i,j \leq N}|v_i(t) - v_j(t)|.
\)
\begin{definition}\label{def:flocking}
We say that the system with particle positions $x_i(t)$ and velocities $v_i(t)$,
$i=1,\dots,N$ and $t\geq 0$, exhibits \emph{asymptotic flocking}
if the spatial and velocity diameters satisfy
\[
   \sup_{t\geq 0} d_X(t) < \infty \quad \mbox{and} \quad \lim_{t\to\infty} d_V(t) = 0.
\]
\end{definition}
We then state our main result in this section on the asymptotic flocking behavior of the system \eqref{main_eq}.
\begin{theorem}\label{thm_main}
Assume $N>2.$
Suppose that the initial data $x_i^0, v_i^0$ are continuous on the time interval $[-\tau_0,0]$
and denote
\(   \label{Rv}
    R_v:= \max_{s\in [-\tau_0, 0]}\max_{1\le i\le N} \vert v_i^0(s)\vert\,.
\)
Moreover, denoted $\beta_N=\frac {N-2}{N-1}$, assume that
\(  \label{ass1}
h(0)d_V(0) + \int_0^{\tau_0} \alpha(s)\lt( \int_{-s}^0 d_V(z)\, \,dz\rt) ds < \beta_N\int_0^{\tau_*} \alpha(s) \int_{d_X(-s) + R_v \tau_0}^\infty \psi(z)\,dzds,
\)
where $d_X$ and $d_V$ denote, respectively, the spatial and velocity diameters defined in \eqref{dXdV}.
Then the solution of the system \eqref{main_eq}--\eqref{IC0} is global in time
and satisfies
\[
\sup_{t\geq 0} d_X(t) < \infty
\]
and
\[
   d_V(t) \leq \max_{s\in [-\tau_0,0]}d_V(s)e^{-\gamma t} \quad \mbox{for } t \geq 0,
\]
for a suitable positive constant $\gamma$ independent of $t$. The constant $\gamma$ can be chosen also independent of $N.$
\end{theorem}
{\begin{remark} {\rm If the influence function $\psi$ is not integrable, i.e., it has a heavy tail, then the right hand side of \eqref{ass1} becomes infinite, and thus the assumption \eqref{ass1} is satisfied for all initial data and time delay $\tau(t)$ satisfying \eqref{tau2}. This is reminiscent of the unconditional flocking condition for the Cucker-Smale type models, see \cite{CFRT, CHL, CH17}. On the other hand, if the influence function $\psi$ is integrable and the weight function $\alpha(s)$ is given by $\alpha(s) = \delta_{\tau}(s)$, then the system \eqref{main_eq} with a constant time delay $\tau(t)\equiv \tau$ becomes the one with discrete time delay:
$$\begin{aligned}
\tot{x_i(t)}{t} &= v_i(t),\qquad i=1,\cdots,N, \quad t >0,\\
\tot{v_i(t)}{t} & =  \sum_{k=1}^N \phi(x_k(t-\tau),x_i(t)) (v_k(t-\tau) - v_i(t)).
\end{aligned}$$
Note the above system is studied in \cite{CH17}. For this system, the assumption \eqref{ass1} reduces to
\[
d_V(0) + \int_{-\tau}^0 d_V(z)\,dz < \beta_N \int_{d_X(-\tau)+R_v \tau}^\infty \psi(z)\,dz.
\]
Since $d_V(z) \geq 0$, the left hand side of the above inequality increases as the size of time delay $\tau>0$ increases. On the other hand, the integrand on the right hand side decreases as $\tau$ increases. This asserts that small size of the time delay provides a  larger set of initial data.
}
\end{remark}
}
\begin{remark}{\rm
Observe that our theorem above gives a flocking result when the number of agents $N$ is greater than two. The result for $N=2$ is trivial for our normalized model.
}\end{remark}

\begin{remark}\label{rmk_2.2}
{\rm
Note that $0 < \beta_N \uparrow 1$ as $N \to \infty.$ Then, this implies that for any $\beta \in (0,1)$ there exists $N_* \in \mathbb{N}$ such that $\beta_N \geq \beta$ for $N \geq N_*$.}
\end{remark}

For the proof of Theorem \ref{thm_main}, we will need several auxiliary results. Inspired by \cite{CH17}, we first show the uniform-in-time bound estimate of the maximum speed of the system \eqref{main_eq}.
\begin{lemma}\label{lem_bdd}
Let $R_v> 0$ be given by \eqref{Rv} and
let $(\bx,\bv)$ be a local-in-time $\mathcal{C}^1$-solution of the system \eqref{main_eq}--\eqref{IC0}.
Then the solution is global in time and satisfies
\[
   \max_{1 \leq i \leq N} |v_i(t)| \leq R_v \quad \mbox{for} \quad t \geq -\tau_0.
\]
\end{lemma}

\begin{proof}
Let us fix $\epsilon>0$ and set
$$S^\epsilon := \{\, t>0\,:\, \max_{1\le i\le N} \vert v_i(s)\vert <R_v+\epsilon \quad \forall\ s\in [0,t)\,\}.$$
By continuity, $S^\epsilon\ne\emptyset.$ Let us denote $T^\epsilon :=\sup S^{\epsilon} > 0.$ We want to show that $T^\epsilon =+\infty.$
Arguing by contradiction, let us suppose $T^\epsilon$ finite. This gives, by continuity,
\begin{equation}\label{Q2}
\max_{1\le i\le N} \vert v_i(T^\epsilon)\vert =R_v + \epsilon.
\end{equation}
From \eqref{main_eq}--\eqref{IC0}, for $t<T^\epsilon$ and $i=1,\dots,N,$ we have
$$
\begin{array}{l}
\displaystyle{
\tot{\vert v_i(t)\vert^2}{t} \le  \frac 2{h(t)} \sum_{k=1}^N \int_{t-\tau(t)}^t \alpha(t-s)\phi(x_k(s),x_i(t)) (\vert v_k(s)\vert \,\vert v_i(t)\vert  - \vert v_i(t)\vert^2)\,ds}
\\
\displaystyle{\hspace{1,5 cm}
\le \frac 2 {h(t)} \sum_{k=1}^N \int_{t-\tau(t)}^t \alpha(t-s)\phi(x_k(s),x_i(t)) \max_{1\le k\le N} \vert v_k(s)\vert \,\vert v_i(t)\vert\, ds  - 2\vert v_i(t)\vert^2}\,.
\end{array}
$$
Note that
\[
\max_{1\le k\le N} \vert v_k(s)\vert \leq R_v + \epsilon \quad \mbox{and} \quad \sum_{k=1}^N\phi(x_k(s),x_i(t)) = 1,
\]
for $s \in [t-\tau(t), t]$ with $t < T^\epsilon$. Thus we obtain
$$
\tot{\vert v_i(t)\vert^2}{t} \le 2 [(R_v+\epsilon) \vert v_i(t)\vert -\vert v_i(t)\vert^2]\,,$$
which gives
\begin{equation}\label{Q1}
\tot{\vert v_i(t)\vert}{t} \le  (R_v+\epsilon) -\vert v_i(t)\vert\,.
\end{equation}
From \eqref{Q1} we obtain
$$
\lim_{t \to {T^\epsilon}^{-}} \ \max_{1 \leq i \leq N}\vert v_i(t)\vert \le e^{-T^\epsilon} \lt(\max_{1 \leq i \leq N}\vert v_i(0) \vert -R_v-\epsilon\rt)+R_v+\epsilon < R_v+\epsilon.
$$
This is in contradiction with \eqref{Q2}. Therefore, $T^\epsilon=+\infty\,.$
Being $\epsilon>0$ arbitrary, the claim is proved.
\end{proof}

In the lemma below, motivated from \cite{CH17, Tan} we derive the differential inequalities for $d_X$ and $d_V$. We notice that the diameter functions $d_X$ and $d_V$ are not $\mc^1$ in general. Thus we introduce the upper Dini derivative to consider the time derivative of these functions: for a given function $F = F(t)$, the upper Dini derivative of $F$ at $t$ is defined by
\[
D^+ F(t) := \limsup_{h \to 0+} \frac{F(t+h) - F(t)}{h}.
\]
Note that the Dini derivative coincides with the usual derivative when the function is differentiable at $t$.
\begin{lemma}\label{lem_sddi}
Let $R_v> 0$ be given by \eqref{Rv} and
let $(\bx,\bv)$ be the global $\mathcal{C}^1$-solution of the system \eqref{main_eq}--\eqref{IC0}
constructed in Lemma \ref{lem_bdd}.
Then, for almost all $t>0$, we have
$$\begin{aligned}
\left\vert D^+ d_X(t)\right\vert   &\leq d_V(t),\cr
D^+ d_V(t) &\leq \frac 1 {h(t)}\int_{t-\tau(t) }^t \alpha(t-s)\Big(1 - \beta_N\psi(d_X(s) + R_v \tau_0)\Big)d_V(s)\,ds - d_V(t).
\end{aligned}$$
\end{lemma}

\begin{proof} The first inequality is by now standard, Then, we concentrate on the second one.
Due to the continuity of the velocity trajectories $v_i(t),$  there is an at most countable system of open, mutually disjoint
intervals $\{\mathcal{I}_\sigma\}_{\sigma\in\N}$ such that
$$
   \bigcup_{\sigma\in\N} \overline{\mathcal{I}_\sigma} = [0,\infty)
$$
and thus for each ${\sigma\in\N}$ there exist indices $i(\sigma)$, $j(\sigma)$
such that
$$
   d_V(t) = |v_{i(\sigma)}(t) - v_{j(\sigma)}(t)| \quad\mbox{for } t\in \mathcal{I}_\sigma.
$$
Then, by using the simplified notation $i:=i(\sigma)$, $j:=j(\sigma)$ (we may assume that $i\ne j$),
we have for every $t\in \mathcal{I}_\sigma$,
\begin{equation}\label{X6}
\begin{array}{l}
\displaystyle{
\frac 12 D^+ d_V^2(t) }\\
\hspace{1 cm} \displaystyle{ = (v_i(t) - v_j(t)) \cdot \lt(\tot{v_i(t)}{t} - \tot{v_j(t)}{t}\rt)}\\
\hspace{1 cm} \displaystyle{ = (v_i(t) - v_j(t)) \cdot \Bigg( \frac 1 {h(t)}\sum_{k=1}^N \int_{t-\tau (t)}^t \alpha(t-s) \phi(x_k(s),x_i(t))v_k(s)\,ds }\\
\displaystyle{
\hspace{5cm} - \frac 1 {h(t)} \sum_{k=1}^N \int_{t-\tau(t)}^t \alpha(t-s)\phi(x_k(s),x_j(t))v_k(s)\,ds\Bigg)}\\
\displaystyle{\qquad\qquad   -  |v_i(t) - v_j(t)|^2.}
\end{array}
\end{equation}
Set
$$\phi_{ij}^k(s,t):=\min\ \left\{ \phi (x_k(s), x_i(t)), \phi (x_k(s), x_j(t))\right\}\quad \mbox{and}\quad \bar\phi_{ij}(s,t):=\sum_{k=1}^N \phi_{ij}^k(s,t)\,.$$
Note that, from the definition \eqref{interact} of $\phi (x_k(s), x_i(t)),$ it results  $\bar\phi_{ij}(s,t)<1$ for all $i\ne j,$ $1\le i,j\le N$ and for all $t>0, \ s\in [t-\tau(t), t].$
Then, we get

\begin{equation}\label{X1}
\begin{array}{l}
\displaystyle{
\frac 1 {h(t)}\sum_{k=1}^N \int_{t-\tau (t)}^t \alpha(t-s) \phi(x_k(s),x_i(t))v_k(s)\,ds}\\
\hspace{1 cm} \displaystyle{- \frac 1 {h(t)} \sum_{k=1}^N \int_{t-\tau(t)}^t \alpha(t-s)\phi(x_k(s),x_j(t))v_k(s)\,ds}\\
\displaystyle{\hspace{1,6 cm}=
\frac 1 {h(t)}\sum_{k=1}^N \int_{t-\tau (t)}^t \alpha(t-s)\left ( \phi(x_k(s),x_i(t)) -\phi_{ij}^k(s,t)\right )v_k(s)\,ds}\\
\displaystyle{\hspace{1,6 cm}
- \frac 1 {h(t)} \sum_{k=1}^N \int_{t-\tau(t)}^t \alpha(t-s)\left (\phi(x_k(s),x_j(t))-\phi_{ij}^k(s,t) \right )v_k(s)\,ds}\\
\displaystyle{\hspace{1,6 cm}=\frac 1 {h(t)}\int_{t-\tau (t)}^t \alpha(t-s) (1-\bar\phi_{ij}(s,t)) \sum_{k=1}^N (a_{ij}^k(s,t) -a_{ji}^k(s,t))v_k(s) \,ds,}
\end{array}
\end{equation}
where
$$ a_{ij}^k(s,t)=\frac {\phi (x_k(s),x_i(t))-\phi_{ij}^k(s,t)}
{1-\bar\phi_{ij}(s,t)},\quad i\ne j,\ 1\le i,j,k\le N\,.
$$
Observe that $a_{ij}^k(s,t)\ge 0$ and $\sum_{k=1}^N a_{ij}^k(s,t)=1.$
Thus, if we consider the convex hull of a finite velocity point set $\{ v_1(t), \dots, v_N(t)\}$ and denote it by $\Omega (t),$ then
$$
 \sum_{k=1}^N a_{ij}^k(s,t) v_k(s)\in\Omega (s) \quad \mbox{\rm for\ all} \ 1\le i \neq j\le N.
 $$
 This gives
 $$\left\vert
 \sum_{k=1}^N \left (a_{ij}^k(s,t)- a_{ji}^k(s,t)\right ) v_k(s)
 \right\vert \le d_V(s),$$
 which, used in \eqref{X1}, implies
 \begin{equation}\label{X2}
\begin{array}{l}
\displaystyle{
\frac 1 {h(t)}\Bigg|\sum_{k=1}^N \int_{t-\tau (t)}^t \alpha(t-s) \phi(x_k(s),x_i(t))v_k(s)\,ds}\\
\hspace{1 cm} \displaystyle{- \sum_{k=1}^N \int_{t-\tau(t)}^t \alpha(t-s)\phi(x_k(s),x_j(t))v_k(s)\,ds}\Bigg|\\
\displaystyle{\hspace{2 cm}\le
\frac 1 {h(t)}\int_{t-\tau (t)}^t \alpha(t-s) (1-\bar \phi_{ij}(s,t)) d_V(s)\, ds\,.}
\end{array}
\end{equation}
Now, by using the first equation in \eqref{main_eq}, we estimate for any $1 \leq i, k \leq N$ and $s \in [t-\tau (t),t]$,
$$\begin{aligned}
|x_k(s) - x_i(t)| &= \lt|x_k(s) - x_i(s) -\int^{t}_{s} \tot{}{t}{x}_i(z)\,dz\rt| \cr
&\leq |x_k(s) - x_i(s)| + \tau_0 \sup_{z \in [t-\tau (t), t]}|v_i(z)|.
\end{aligned}$$
Then, Lemma \ref{lem_bdd} gives
\[
   |x_k(s) - x_i(t)| \leq d_X(s) + R_v\tau_0, \quad \mbox{for} \quad s \in [t-\tau(t),t],
\]
and due to the monotonicity property of the influence function $\psi,$ for $k\ne i,$ we deduce
\begin{equation}\label{X3}
\phi(x_k(s),x_i(t)) \geq \frac{\psi(d_X(s) + R_v \tau_0)}{N-1}.
\end{equation}
On the other hand, we find
$$ \bar \phi_{ij}=\sum_{k=1}^N \phi_{ij}^k= \sum_{k\ne i,j} \phi_{ij}^k+\phi_{ij}^i + \phi_{ij}^j =\sum_{k\ne i,j} \phi_{ij}^k\,.$$
Then, from \eqref{X3}, we obtain
$$ \bar \phi_{ij} (s,t)\ge \frac {N-2}{N-1} \psi (d_X(s)+R_v \tau_0)=\beta_N \psi (d_X(s)+R_v \tau_0)\,.$$
Using the last estimate in \eqref{X2}, we have
\begin{equation*}
\begin{array}{l}
\displaystyle{
\frac 1 {h(t)}\Bigg|\sum_{k=1}^N \int_{t-\tau (t)}^t \alpha(t-s) \phi(x_k(s),x_i(t))v_k(s)\,ds}\\
\hspace{1 cm} \displaystyle{- \frac 1 {h(t)} \sum_{k=1}^N \int_{t-\tau(t)}^t \alpha(t-s)\phi(x_k(s),x_j(t))v_k(s)\,ds}\Bigg|\\
\displaystyle{\hspace{2 cm}\le
\frac 1 {h(t)}\int_{t-\tau (t)}^t \alpha(t-s) \Big(1 - \beta_N\psi(d_X(s) + R_v \tau_0)\Big)d_V(s)\,ds,}
\end{array}
\end{equation*}
that, used in \eqref{X6}, concludes the proof.
\end{proof}
\begin{lemma}\label{lem_gron}
Let $u$ be a nonnegative, continuous and piecewise $\mc^1$-function satisfying,
for some constant $0 < a < 1$, the differential inequality
\bq\label{diff_ineq}
   \tot{}{t}  u(t) \leq \frac a {h(t)} \int_{t-\tau (t)}^t \alpha(t-s)u(s)\,ds - u(t) \qquad\mbox{for almost all } t>0.
\eq
Then we have
\[
   u(t) \leq \sup_{t\in [-\tau_0, 0]} u(s)\, e^{-\gamma t} \quad \mbox{for all } t \geq 0,
\]
with {$\gamma = 1 - \tau_0^{-1}H(a\tau_0 \exp(\tau_0))$}, where $H$ is the product logarithm function, i.e., $H$ satisfies $z = H(z) \exp(H(z))$ for any $z \geq 0$.
\end{lemma}

\begin{proof} Note that the differential inequality \eqref{diff_ineq} implies
\[
   \tot{}{t}  u(t) \leq a \sup_{s\in [t-\tau_0, t]} u(s)-u(t).
\]
Then, the result follows from Halanay inequality (see e.g. \cite[p. 378]{Halanay}).
\end{proof}

We are now ready to proceed with the proof of Theorem \ref{thm_main}.

\begin{proof}[Proof of Theorem \ref{thm_main}]
For $t > 0$, we introduce the following Lyapunov functional for the system \eqref{main_eq}--\eqref{IC0}:
$$\begin{aligned}
   \Lyap(t) &:= h(t)d_V(t) + \beta_N \int_0^{\tau(t)} \alpha(s)\, { \lt\vert \int_{d_X(-s) + R_v \tau_0}^{d_X(t - s) + R_v\tau_0} \psi(z)\,d z\rt\vert }\,ds \cr
&\quad   + \int_0^{\tau(t)} \alpha(s)\lt( \int_{-s}^0 d_V(t + z)\, \,dz\rt) ds \cr
& =: h(t)d_V(t)+\Lyap_1(t) +\Lyap_2 (t),
\end{aligned}$$
where $R_v$ is given by \eqref{Rv} and the diameters $d_X(t)$, $d_V(t)$ are defined in \eqref{dXdV}. We first estimate $\Lyap_1$ as
{\begin{align}\label{Q3}
\begin{aligned}
D^+\Lyap_1(t)&= \beta_N\left\{\tau^\prime (t)\alpha (\tau(t))\,\lt\vert\int_{d_X(-\tau(t)) + R_v \tau_0}^{d_X(t - \tau(t)) + R_v\tau_0} \psi(z)\,d z\rt\vert  \right. \cr
&\quad\left. +  \int_0^{\tau(t)} \alpha(s) \, \mbox{ sign\,} \lt( \int_{d_X(-s) + R_v \tau_0}^{d_X(t - s) + R_v\tau_0} \psi(z)\,d z\rt)\psi (d_X(t - s) + R_v\tau_0) D^+d_X(t-s)\, ds\right\}\cr
&\leq \beta_N\int_0^{\tau(t)} \alpha(s) \psi (d_X(t - s) + R_v\tau_0) d_V(t-s)\, ds,
\end{aligned}
\end{align}
for almost all $t \geq 0$, due to \eqref{tau2} and the first inequality in Lemma \ref{lem_sddi}. Here $\mbox{sgn\,}(\cdot)$ is the signum function defined by
\[
\mbox{sgn\,}(x) := \left\{ \begin{array}{ll}
 -1 & \textrm{if $x< 0$},\\
\ 0 & \textrm{if $x= 0$},\\
\  1 & \textrm{if $x> 0$}.
  \end{array} \right.
\]
}
Analogously, we also get
\begin{align}\label{Q4}
\begin{aligned}
D^+\Lyap_2(t) &\leq \int_0^{\tau(t)} \alpha(s) \lt( d_V(t)- d_V(t-s)\rt) \, ds \cr
&=h(t)d_V(t) - \int_0^{\tau(t)} \alpha(s) d_V(t-s) \, ds,
\end{aligned}
\end{align}
for almost all $t \geq 0$. Then, from Lemma \ref{lem_sddi}, \eqref{Q3} and \eqref{Q4}, we have, for almost all $t>0$,
$$\begin{aligned}
D^+\Lyap(t) &\leq h'(t)d_V(t) \cr
&\quad + \int_{0}^{\tau(t)} \alpha(s)\Big (1 - \beta_N\psi(d_X(t-s) + R_v\tau_0)\Big )d_V(t-s)\,ds - h(t)d_V(t)\cr
   &\quad + \beta_N\int_0^{\tau(t)} \alpha(s) \psi (d_X(t - s) + R_v\tau_0) d_V(t-s)\, ds\cr
   &\quad + h(t)d_V(t) - \int_0^{\tau(t)} \alpha(s) d_V(t-s) \, ds\cr
   &=h'(t) d_V(t),
\end{aligned}$$
On the other hand, since $h^\prime(t) = \alpha(\tau(t))\tau^\prime(t) \leq 0$, we have $D^+\left (h(t)\Lyap(t)\right )\le 0,$ namely
\begin{align}\label{Q6}
\begin{aligned}
h(t) d_V(t) &+ \beta_N\int_0^{\tau(t)} \alpha(s) \, \lt\vert\int_{d_X(-s) + R_v\tau_0}^{d_X(t - s) + R_v \tau_0} \psi(z)\,d z\rt\vert \, ds \cr
&\quad + \int_0^{\tau(t)} \alpha(s)\lt( \int_{-s}^0 d_V(t + z)\, \,dz\rt) ds \cr
& \leq h(0) d_V(0) + \int_0^{\tau_0} \alpha(s)\lt( \int_{-s}^0 d_V(z)\, \,dz\rt) ds.
\end{aligned}
\end{align}
Moreover, it follows from the assumption \eqref{ass1} that there exists a positive constant $d_* > d_X(-s) + R_v\tau_0  $ such that
$$
h(0) d_V(0) + \int_0^{\tau_0} \alpha(s)\lt( \int_{-s}^0 d_V(z)\, \,dz\rt) ds \leq \beta_N \int_0^{\tau_*} \alpha(s) \int_{d_X(-s) + R_v \tau_0}^{d_*} \psi(z)\,dzds\,.
$$
This, together with \eqref{Q6} and \eqref{tau1},  implies
\begin{equation}\label{serve}
\begin{array}{l}
\displaystyle{
h(t) d_V(t) + \int_0^{\tau(t)} \alpha(s)\lt( \int_{-s}^0 d_V(t + z)\, \,dz\rt) ds} \\
\displaystyle{\quad \leq\beta_N\lt \{\int_0^{\tau_*} \alpha(s) \int_{d_X(-s) + R_v \tau_0}^{d_*} \psi(z)\,dzds
-\int_0^{\tau_*}\alpha (s)\,  \lt \vert \int_{d_X(-s) + R_v\tau_0}^{d_X(t - s) + R_v \tau_0} \psi(z)\,d z\rt \vert \, ds\rt \}}\\ \\
\displaystyle\quad = \beta_N\int_0^{\tau_*} \alpha(s) \left\{\int_{d_X(-s) + R_v \tau_0}^{d_*} \psi(z)\,dz- \left\vert \int_{d_X(-s) + R_v \tau_0}^{d_X(t-s) + R_v \tau_0} \psi(z)\,dz\right\vert \right\} ds.
\end{array}
\end{equation}
{Now, observe that, if $d_X(t-s)\geq d_X(-s),$ then
\begin{equation}\label{Serve1}
\begin{array}{l}
\displaystyle{
\int_0^{\tau_*} \alpha(s) \left\{
\int_{d_X(-s) + R_v \tau_0}^{d_*} \psi(z)\,dz- \left\vert \int_{d_X(-s) + R_v \tau_0}^{d_X(t-s) + R_v \tau_0} \psi(z)\,dz\right\vert
\right\} ds}\\ \\
\hspace{2 cm}
\displaystyle{
= \int_0^{\tau_*} \alpha(s) \int_{d_X(t-s) + R_v \tau_0}^{d_*} \psi(z)\,dz\,ds}\,.
\end{array}
\end{equation}
Similarly, when $d_X(t-s)<d_X(-s),$ then
\begin{equation}\label{Serve2}
\begin{array}{l}
\displaystyle{
\int_0^{\tau_*} \alpha(s) \left\{\int_{d_X(-s) + R_v \tau_0}^{d_*} \psi(z)\,dz- \left\vert \int_{d_X(-s) + R_v \tau_0}^{d_X(t-s) + R_v \tau_0} \psi(z)\,dz\right\vert \right\} ds}\\ \\
\hspace{1 cm}\displaystyle{
\le \int_0^{\tau_*} \alpha(s) \int_{d_X(-s) + R_v \tau_0}^{d_*} \psi(z)\,dz\,ds
\le\int_0^{\tau_*} \alpha(s) \int_{d_X(t-s) + R_v \tau_0}^{d_*} \psi(z)\,dz\,ds\,.}
\end{array}
\end{equation}
Thus, from
\eqref{serve}, \eqref{Serve1} and \eqref{Serve2} we deduce that
\begin{equation*}
\begin{array}{l}
\displaystyle{
h(t) d_V(t) + \int_0^{\tau(t)} \alpha(s)\lt( \int_{-s}^0 d_V(t + z)\, \,dz\rt) ds}
\\
\hspace{2 cm} \displaystyle{
\le\beta_N
 \int_0^{\tau_*} \alpha(s) \int_{d_X(t-s) + R_v \tau_0}^{d_*} \psi(z)\,dz\,ds\,.}
 \end{array}
 \end{equation*}
}
{Note that,  for $s \in [0,\tau(t)]$,
\[
d_X(t-s)= d_X(t) +\int_t^{t-s} D^+ d_X(z) dz \le d_X(t) +2R_v\tau_0,
\]
due to  Lemma \ref{lem_bdd} and the first inequality of Lemma \ref{lem_sddi}.
Analogously, we also find for $s \in [0,\tau(t)]$ that
\[
d_X(t) = d_X(t-s) + \int_{t-s}^t D^+ d_X(z)\,dz \leq d_X(t-s) + 2R_v\tau_0.
\]
}
This gives
\bq\label{est_dx}
d_X(t-s) - 2R_v\tau_0 \le d_X(t) \le d_X(t-s) + 2R_v\tau_0, \quad \mbox{\rm for}\quad  s \in [0,\tau(t)]\,.
\eq
Thus we get
$$\begin{aligned}
\int_0^{\tau_*} \alpha(s) \int_{d_X(t -s) + R_v \tau_0}^{d_*} \psi(z)\,dzds &\leq \int_0^{\tau_*} \alpha(s) \int_{\max \{d_X(t)-R_v\tau_0, 0\}}^{d_*} \psi(z)\,dzds\cr
&\leq h(0)  \int_{\max \{d_X(t)-R_v\tau_0, 0\}}^{d_*} \psi(z)\,dz.
\end{aligned}$$
Combining this and \eqref{serve}, we obtain
\[
h(t)d_V(t) + \int_0^{\tau (t)} \alpha(s)\lt( \int_{-s}^0 d_V(t+z)\, \,dz\rt) ds \leq h(0)\beta_N  \int_{\max \{d_X(t)-R_v\tau_0, 0\}}^{d_*} \psi(z)\,dz.
\]
Since the left hand side of the above inequality is positive, we have
\[
d_X(t) \leq d_* +R_v\tau_0\quad \mbox{for} \quad t \geq 0.
\]
We then again use \eqref{est_dx} to find
\[
d_X(t-s) + R_v\tau_0 \leq d_X(t) + 3R_v \tau_0 \leq d_* + 4R_v \tau_0,
\]
for $s \in [0,\tau(t)]$ and $t \geq 0$. Hence, by Lemma \ref{lem_sddi} together with the monotonicity of $\psi$ we have
\[
D^+ d_V(t) \leq \frac {(1 - \psi_*\beta_N)} {h(t)} \int_{t-\tau(t)}^t \alpha(t-s) d_V(s)\,ds -  d_V(t),
\]
for almost all $t>0$, where $\psi_* = \psi(d_* + 4R_v\tau_0)$. We finally apply Lemma \ref{lem_gron} to complete the proof.
Note that, in order to have an exponential decay rate of $d_V$ independent of $N,$ it is sufficient to observe that $\beta_N\ge 1/2$ for $N\ge 3.$
\end{proof}

%
%
%
%
\section{A delayed Vlasov alignment equation}\label{MF}
In this section, we are interested in the behavior of solutions to the particle system \eqref{main_eq} as the number of particles $N$ goes to infinity. At the formal level, we can derive the following delayed Vlasov alignment equation from \eqref{main_eq} as $N \to \infty$:
\bq\label{main_kin}
\pa_t f_t + v \cdot \nabla_x f_t + \nabla_v \cdot \displaystyle \lt(\frac{1}{h(t)}\int_{t - \tau(t)}^t \alpha
(t-s)F[f_s]\,dsf_t\rt) = 0, \ (x,v) \in \R^d \times \R^d,  \ t > 0,
\eq
with the initial data:
\[
f_s(x,v) = g_{s}(x,v), \quad (x,v) \in \R^d \times \R^d, \quad s \in [-\tau_0,0],
\]
where $f_t=f_t(x,v)$ is the one-particle distribution function on the phase space $\R^d \times \R^d$. Here the interaction term $F[f_s]$ is given by
\[
F[f_s](x,v) :=  \frac{\displaystyle \int_{\R^d \times \R^d} \psi(|x-y|)(w-v)f_s(y,w)\,dydw}{\displaystyle \int_{\R^d \times \R^d} \psi(|x-y|)f_s(y,w)\,dydw}.
\]
For the equation \eqref{main_kin}, we provide the global-in-time existence and uniqueness of measure-valued solutions and mean-field limits from \eqref{main_eq} based on the stability estimate. We also establish the large-time behavior of measure-valued solutions showing the velocity alignment.

\subsection{Existence of measure-valued solutions} In this part, we discuss the global existence and uniqueness of measure-valued solutions to the equation \eqref{main_kin}. For this, we first define a notion of weak solutions in the definition below.
\begin{definition}\label{def_weak}
For a given $T >0$, we call $f_t \in \mc([0,T); \P_1(\R^d \times \R^d))$ a \emph{measure-valued solution} of the equation \eqref{main_kin}
on the time interval $[0,T)$, subject to the initial datum $g_s\in \mc([-\tau_0,0]; \P_1(\R^d \times \R^d))$,
if for all compactly supported test functions $\xi \in \mc_c^\infty(\R^d \times \R^d \times [0,T))$,
$$\begin{aligned}
&\int_0^T \int_{\R^d \times \R^d} f_t\lt(\pa_t \xi + v \cdot \nabla_x \xi + \frac{1}{h(t)}\int_{t - \tau(t)}^t \alpha(t-s)F[f_s]\,ds\cdot\nabla_v\xi\rt)dxdvdt \cr
&\quad + \int_{\R^d \times \R^d} g_0(x,v)\xi(x,v,0)\, dxdv = 0,
\end{aligned}$$
where $\P_1(\R^d \times \R^d)$ denotes the set of probability measures on the phase space $\R^d \times \R^d$ with bounded first-order moment and we adopt the notation $f_{t-\tau_0}:\equiv g_{t-\tau_0}$ for $t \in [0,\tau_0]$.
\end{definition}
We next introduce the $1$-Wasserstein distance.
\begin{definition}\label{defdp}
Let $\rho^1, \rho^2\in \P_1(\R^{d})$ be two probability measures on $\R^{d}$. Then $1$-Waserstein distance between $\rho^1$ and $\rho^2$ is defined as
\begin{equation*}\label{d1}
\mathcal{W}_1(\rho^1,\rho^2) := \inf_{\pi\in\Pi(\rho^1,\rho^2)} \int_{\R^d \times \R^d} |x-y|  \d\pi(x,y) ,
\end{equation*}
where $\Pi(\rho^1,\rho^2)$ represents the collection of all probability measures on $\R^d \times \R^d$ with marginals $\rho^1$ and $\rho^2$ on the first and second factors, respectively.
\end{definition}

\begin{theorem}\label{main_thm2}
Let the initial datum $g_t \in  \mc([-\tau_0,0]; \P_1(\R^d \times \R^d))$
and assume that there exists a constant $R > 0$ such that
\[
  \mbox{supp } g_t \subset B^{2d}(0,R) \qquad\mbox{for all } t \in [-\tau_0,0],
\]
where $B^{2d}(0,R)$ denotes the ball of radius $R$ in $\R^d \times \R^d$, centered at the origin.

Then for any $T > 0$, the delayed Vlasov alignment equation \eqref{main_kin} admits a unique measure-valued solution $f_t\in  \mc([0,T); \P_1(\R^d \times \R^d))$
in the sense of Definition \ref{def_weak}.
\end{theorem}
\begin{proof} The proof can be done by using a similar argument as in \cite[Theorem 3.1]{CH17}, thus we shall give it rather concisely. Let $f_t \in \mc([0,T];\P_1(\R^d \times \R^d))$ be such that
\[
  \mbox{supp } f_t \subset B^{2d}(0,R) \qquad\mbox{for all } t \in [0,T],
\]
for some positive constant $R > 0$. Then we use the similar estimate as in \cite[Lemma 3.1]{CH17} to have that there exists a constant $C > 0$ such that
\[
   |F[f_t](x,v)| \leq C \quad \mbox{and} \quad |F[f_t](x,v) - F[f_t](\tilde x,\tilde v)| \leq C\bigl( |x-\tilde x| + |v - \tilde v|\bigr),
\]
for $(x,v) \in B^{2d}(0,R),\; t\in [0,T]$. Subsequently, this gives
\[
\lt|\frac{1}{h(t)}\int_{t - \tau(t)}^t \alpha
(t-s)F[f_s]\,ds\rt| \leq C
\]
and
\[
\lt|\frac{1}{h(t)}\int_{t - \tau(t)}^t \alpha
(t-s)\lt(F[f_s](x,v) - F[f_s](\tilde x,\tilde v)\rt)ds\rt| \leq C
\]
for $(x,v) \in B^{2d}(0,R),\; t\in [0,T]$. This, together with an application of \cite[Theorem 3.10]{CCR}, provides the local-in-time existence and uniqueness of measure-valued solutions
to the equation \eqref{main_kin} in the sense of Definition \ref{def_weak}. On the other hand, once we obtain the growth estimates of support of $f_t$ in both position and velocity, these local-in-time solutions can be extended to the global-in-time ones. Thus in the rest of this proof, we provide the support estimate of $f_t$ in both position and velocity. We set
\[
   R_X[f_t] := \max_{x \,\in\, \overline{\mbox{\footnotesize supp}_x f_t}}|x|,  \qquad
   R_V[f_t] := \max_{v \,\in\, \overline{\mbox{\footnotesize supp}_v f_t}}|v|,
\]
   for $t \in [0,T]$, where supp$_x f_t$ and supp$_v f_t$ represent $x$- and $v$-projections of supp$f_t$, respectively. We also set
\(  \label{supp_diams}
   R^t_X := \max_{-\tau_0 \leq s \leq t}R_X[f_s], \qquad  R^t_V := \max_{-\tau_0 \leq s \leq t}R_V[f_s].
\)
We first construct the system of characteristics
$Z(t;x,v) := (X(t;x,v),V(t;x,v)): [0,\tau] \times \R^d \times \R^d \to \R^d \times \R^d$
associated with \eqref{main_kin},
\begin{align}\label{tra_1}
\begin{aligned}
    \tot{X(t;x,v)}{t} &= V(t;x,v),\\
    \tot{V(t;x,v)}{t} &= \frac{1}{h(t)}\int_{t- \tau(t)}^t \alpha(t-s)F[f_s]\lt(Z(t;x,v)\rt)\,ds,
\end{aligned}
\end{align}
where we again adopt the notation $f_{t-\tau_0}:\equiv g_{t-\tau_0}$ for $t \in [0,\tau_0]$.
The system \eqref{tra_1} is considered subject to the initial conditions
\(  \label{tra_1_IC}
   X(0;x,v) = x,\qquad V(0;x,v) = v,
\)
for all $(x,v)\in\R^d \times \R^d$. Note that the characteristic system \eqref{tra_1}--\eqref{tra_1_IC} is well-defined since the $F[f_s]$ is locally both bounded and Lipschitz. We can rewrite the equation for $V$ as
$$\begin{aligned}
  \frac{dV(t;x,v)}{dt} &= \frac{1}{h(t)}\int_{t-\tau(t)}^t \alpha(t-s) \lt(\frac{\int_{\R^d \times \R^d} \psi(|X(t;x,v) - y|) w\,df_s(y,w)}{\int_{\R^d \times \R^d} \psi(|X(t;x,v) - y|)\,df_s(y,w)} \rt)ds - V(t;x,v)\cr
 &= \frac{1}{h(t)}\int_{0}^{\tau(t)} \alpha(s) \lt(\frac{\int_{\R^d \times \R^d} \psi(|X(t;x,v) - y|) w\,df_{t-s}(y,w)}{\int_{\R^d \times \R^d} \psi(|X(t;x,v) - y|)\,df_{t-s}(y,w)} \rt)ds - V(t;x,v).
\end{aligned}$$
Then, arguing as in the proof of Lemma \ref{lem_bdd}, we get
\[
   \frac{d|V(t)|}{dt} \leq R^t_V - |V(t)|,
\]
due to \eqref{supp_diams}. Using again a similar argument as in the proof of Lemma \ref{lem_bdd} and the comparison lemma, we obtain
\[
   R^t_V \leq R^0_V \quad \mbox{for} \quad t \geq 0,
\]
which further implies $R^t_X \leq R^0_X + t R^0_V$ for $t \geq 0$. This completes the proof.
\end{proof}

\subsection{Stability \& mean-field limit} In this subsection, we discuss the rigorous derivation of the delayed Vlasov alignment equation \eqref{main_kin} from the particle system \eqref{main_eq} as $N \to \infty$. For this, we first provide the stability of measure-valued solutions of \eqref{main_kin}.

\begin{theorem}\label{main_thm3}
Let $f^i_t \in \mc([0, T);\P_1(\R^d \times \R^d))$, $i=1,2$, be two measure-valued solutions of \eqref{main_kin} on the time interval $[0, T]$,
subject to the compactly supported initial data $g^i_s\in \mc([-\tau_0, 0];\P_1(\R^d \times \R^d))$.
Then there exists a constant $C=C(T)$ such that
\[
\mathcal{W}_1(f^1_t,f^2_t) \leq C \max_{s\in[-\tau_0,0]} \mathcal{W}_1(g^1_s,g^2_s) \quad \mbox{for} \quad t \in [0,T).
\]
\end{theorem}
\begin{proof}Again, the proof is very similar to \cite[Theorem 3.2]{CH17}, see also \cite{CCH, CCHS}. Indeed, we can obtain
\[
\frac{d}{dt}\mathcal{W}_1(f^1_t,f^2_t) \leq C \lt( \mathcal{W}_1(f^1_t,f^2_t) + \frac{1}{h(t)}\int_{t - \tau(t)}^t \alpha(t-s)\mathcal{W}_1(f^1_s,f^2_s)\,ds \rt).
\]
Then we have
\[
\mathcal{W}_1(f^1_t,f^2_t) \leq e^{2CT}\max_{s\in[-\tau_0,0]} \mathcal{W}_1(g^1_s,g^2_s),
\]
for $t \in [0,T)$.
\end{proof}
\begin{remark} Since the empirical measure $ f^N_t(x,v) := \frac{1}{N}\sum_{i=1}^N  \delta_{(x^N_i(t), v^N_i(t))}(x,v)$ associated to the $N$-particle system \eqref{main_eq} with the initial data $g^N_s(x,v) := \frac1N\sum_{i=1}^N  \delta_{(x_i^0(s), v_i^0(s))}(x,v)$ for $s\in[-\tau_0,0]$, is a solution to the equation \eqref{main_kin}, we can use Theorem \ref{main_thm3} to have the following mean-field limit estimate:
\[
\sup_{t \in [0,T)}\mathcal{W}_1(f_t,f^N_t) \leq C \max_{s\in[-\tau_0,0]} \mathcal{W}_1(g_s,g^N_s),
\]
where $C$ is a positive constant independent of $N$.
\end{remark}

\subsection{Asymptotic behavior of the delayed Vlasov alignment equation \eqref{main_kin}}
In this part, we provide the asymptotic behavior of solutions to the equation \eqref{main_kin} showing the velocity alignment under suitable assumptions on the initial data. For this, we first define the position- and velocity-diameters for a compactly supported measure $g\in\P_1(\R^d \times \R^d)$,
\[
   d_X[g] := \mbox{diam}\lt({\mbox{supp}_x g}\rt), \qquad  d_V[g] := \mbox{diam}\lt({\mbox{supp}_v g}\rt),
\]
where supp$_x f$ denotes the $x$-projection of supp$f$ and similarly for supp$_v f$.

\begin{theorem}\label{main_thm4}
Let  $T>0$ and $f_t \in \mc([0, T);\P_1(\R^d \times \R^d))$ be a weak solution of \eqref{main_kin} on the time interval $[0, T)$ with compactly supported initial data $g_s\in \mc([-\tau_0, 0];\P_1(\R^d \times \R^d))$. Furthermore, we assume that
\bq\label{kin_flocking_cond}
   h(0)d_V[g_0]  + \int_0^{\tau_0} \alpha(s) \lt(\int_{-s}^0 d_V[g_z]\, dz\rt) ds <  \int_0^{\tau_*} \alpha(s) \lt( \int_{d_X[g_{-s}] + R_V^0 \tau_0}^\infty \psi(z)\, dz\rt)ds.
\eq
Then the weak solution $f_t$ satisfies
\[
   d_V[f_t] \leq \left( \max_{s\in[-\tau_0,0]} d_V[g_s] \right) e^{-Ct} \quad \mbox{for } t \geq 0,
   \qquad \sup_{t\geq 0} d_X[f_t] < \infty,
\]
where $C$ is a positive constant independent of $t$.
\end{theorem}
Let us point out that the flocking estimate at the particle level, see Section \ref{sec_par} and Remark \ref{rmk_2.2}, is independent of the number of particles, thus we can directly use the same strategy as in \cite{CFRT, CH17, Ha-Liu}. However, we provide the details of the proof for the completeness.
\begin{proof}[Proof of Theorem \ref{main_thm4}] We consider an empirical measure $\{g^N_s\}_{N\in\N}$, which is a family of $N$-particle approximations of $g_s$, i.e.,
\[
   g^N_s(x,v) := \frac1N\sum_{i=1}^N  \delta_{(x_i^0(s), v_i^0(s))}(x,v) \qquad\mbox{for } s\in[-\tau_0,0],
\]
where the $x_i^0, v_i^0 \in \mc([-\tau_0,0];\R^d)$ are chosen such that
\[
   \max_{s\in[-\tau_0,0]} \mathcal{W}_1(g^N_s,g_s) \to 0 \quad\mbox{as}\quad N\to\infty.
\]
Note that we can choose $(x_i^0, v_i^0)_{i=1,\dots,N}$ such that the condition \eqref{ass1} is satisfied uniformly in $N\in\N$ due to the assumption \eqref{kin_flocking_cond}. Let us denote by $(x^N_i,v^N_i)_{i=1,\dots,N}$ the solution of the $N$-particle system \eqref{main_eq}--\eqref{IC0} subject to the initial data $(x_i^0,v_i^0)_{i=1,\dots,N}$ constructed above. Then it follows from Theorem \ref{thm_main} that there exists a positive constant $C_1>0$ such that
\[
   d_V(t) \leq \left( \max_{s\in[-\tau_0,0]} d_V(s) \right) e^{-C_1 t} \quad \mbox{for } t \geq 0,
\]
with the diameters $d_V$, $d_X$ defined in \eqref{dXdV}, where $C_1>0$ is independent of $t$ and $N$. Note that the empirical measure
\[
   f^N_t(x,v) := \frac{1}{N}\sum_{i=1}^N  \delta_{(x^N_i(t), v^N_i(t))}(x,v)
\]
is a measure-valued solution of the delayed Vlasov alignment equation \eqref{main_kin} in the sense of Definition \ref{def_weak}. On the other hand, by Theorem \ref{main_thm3}, for any fixed $T>0$,  we have the following stability estimate
\[
 \mathcal{W}_1  (f_t,f^N_t) \leq C_2 \max_{s\in[-\tau_0,0]} \mathcal{W}_1(g_s,g^N_s) \quad \mbox{for} \quad t \in [0,T),
\]
where the constant $C_2>0$ is independent of $N$. This yields that sending $N\to\infty$ implies $d_V[f_t] = d_V(t)$ on $[0,T)$ for any fixed $T>0$. Thus we have
\[
   d_V[f_t] \leq \left( \max_{s\in[-\tau_0,0]} d_V[g_s] \right) e^{-C_1 t} \quad \mbox{for} \quad t \in[0,T).
\]
Since the uniform-in-$t$ boundedness of $d_X[f_t]$ just follows from the above exponential decay estimate, we conclude the desired result.
\end{proof}

%
%
%
%

\section*{Acknowledgments}
The first author was supported by National Research Foundation of Korea(NRF) grant funded by the Korea government(MSIP) (No. 2017R1C1B2012918 and 2017R1A4A1014735) and POSCO Science Fellowship of POSCO TJ Park Foundation. The research of the second author was  partially supported by the GNAMPA 2018  project {\sl Analisi e controllo di modelli differenziali non lineari} (INdAM).

%
%
%
%


\begin{thebibliography}{10}

\bibitem{Albi}
{\sc G. Albi, M. Herty and L. Pareschi},
{\em Kinetic description of optimal control problems and applications to opinion consensus},
{Commun.  Math. Sci.,} 13, (2015), 1407--1429.

\bibitem{Bellomo}
{\sc N. Bellomo, M. A. Herrero and A. Tosin},
{\em On the dynamics of social conflict: Looking for
the Black Swan},
{Kinet. Relat. Models}, 6, (2013), 459--479.

\bibitem{Borzi}
{\sc A. Borz\`i and S. Wongkaew},
{\em Modeling and control through leadership of a refined flocking system},
{Math. Models Methods Appl. Sci.}, 25, (2015), 255--282.

\bibitem{CCR}
{\sc J. Ca\~nizo, J. Carrillo, J. Rosado},
{\em A well-posedness theory in measures for some kinetic models of collective motion},
{Math. Models Methods Appl. Sci.,} 21, (2011), 515--539.

\bibitem{Caponigro}
{\sc M. Caponigro, M. Fornasier, B. Piccoli and E. Tr\'elat},
{\em Sparse stabilization and control of alignment models},
{Math. Models Methods Appl. Sci.,} 25, (2015), 521--564.

\bibitem{CCP} {\sc J. Carrillo, Y.-P. Choi and S. P\'erez},
{\em A review an attractive-repulsive hydrodynamics for consensus in collective behavior}, in N. Bellomo, P. Degond, and E. Tamdor (Eds.),
{\em Active Particles Vol.I: Advances in Theory, Models, Applications}, Series: Modelling and Simulation in Science and Technology, Birkh\"auser Basel, (2017), 259--298.

\bibitem{CCH} {\sc J. A. Carrillo, Y.-P. Choi and M. Hauray},
{\em The derivation of swarming models: Mean-field limit and Wasserstein distances}, in A. Muntean, F. Toschi (Eds.),
{\em Collective Dynamics from Bacteria to Crowds: An Excursion Through Modeling, Analysis and Simulation}, Series: CISM International Centre for Mechanical Sciences, 533, Springer, (2014), 1--46.

\bibitem{CCH2} {\sc J. A. Carrillo, Y.-P. Choi and M. Hauray},
{\em Local well-posedness of the generalized Cucker-Smale model with singular kernels},
{ESAIM Proc.}, 47, (2014), 17--35.

\bibitem{CCHS} {\sc J. Carrillo, Y.-P. Choi, M. Hauray and S. Salem},
{\em Mean-field limit for collective behavior models with sharp sensitivity regions},
{J. Eur. Math. Soc.}, 21, (2019), 121--161.

\bibitem{CCMP}
{\sc J. A. Carrillo, Y.-P. Choi, P. B. Mucha and J. Peszek},
{\em Sharp conditions to avoid collisions in singular Cucker-Smale interactions},
{Nonlinear Anal.-Real World Appl.}, 37, (2017), 317--328.

\bibitem{CFRT}
{\sc J. Carrillo, M. Fornasier, J. Rosado and G. Toscani},
{\em Asymptotic Flocking Dynamics for the kinetic Cucker-Smale model},
{SIAM J. Math. Anal.}, 42, (2010), 218--236.

\bibitem{CHL} {\sc Y.-P. Choi, S.-Y. Ha and Z. Li},
{\em Emergent dynamics of the Cucker-Smale flocking model and its variants}, in N. Bellomo, P. Degond, and E. Tamdor (Eds.),
{\em Active Particles Vol.I: Advances in Theory, Models, Applications}, Series: Modelling and Simulation in Science and Technology, Birkh\"auser Basel, (2017), 299--331.

\bibitem{CH17} {\sc Y.-P. Choi and J. Haskovec}, {\em Cucker-Smale model with normalized communication weights and time delay}, { Kinet. Relat. Models},10, (2017), 1011--1033.

\bibitem{CH19} {\sc Y.-P. Choi and J. Haskovec}, {\em Hydrodynamic Cucker-Smale model with normalized communication weights and time delay}, {SIAM J. Math. Anal.}, to appear.

\bibitem{CL} {\sc Y.-P. Choi and and Z. Li},
{\em Emergent behavior of Cucker-€"Smale flocking particles with heterogeneous time delays},
\newblock {Appl. Math. Lett.}, 86, (2018), 49--56.

\bibitem{CKPP}
{\sc Y.-P. Choi, D. Kalise, J. Peszek and A. A. Peters},
{\em A collisionless singular Cucker-Smale model with decentralized formation control}, preprint, arXiv:1807.05177.

\bibitem{Couzin}
{\sc I. Couzin, J. Krause, N. Franks and S. Levin},
{\em Effective leadership and decision making in
animal groups on the move},
{Nature}, 433, (2005), 513--516.

\bibitem{Cristiani}
{\sc E. Cristiani, B. Piccoli and A. Tosin},
{\em Multiscale modeling of granular flows with application
to crowd dynamics},
{Multiscale Model. Simul.}, 9, (2011), 155--182.

\bibitem{CuckerDong}
{\sc F. Cucker and J.G. Dong},
{\em A general collision-avoiding flocking framework},
{IEEE Trans. Automat. Cont.},
56, (2011), 1124--1129.


\bibitem{CuckerMordecki}
{\sc F. Cucker and E. Mordecki},
{\em Flocking in noisy environments},
{J. Math. Pures Appl.}, 89, (2008), 278--296.

\bibitem{CS1}
{\sc F. Cucker and S. Smale},
{\em Emergent behaviour in flocks},
{IEEE Trans. Automat. Contr.,} 52, (2007), 852--862.

\bibitem{CS2}
{\sc F. Cucker  and S. Smale},
{\em On the mathematics of emergence},
{Japanese Journal of Mathematics}, 2, (2007), 197--227.


\bibitem{EHS} {\sc R. Erban, J. Haskovec, and Y. Sun},
{\em A Cucker-Smale model with noise and delay},
SIAM J. Appl. Math, 76, (2016), 1535--1557.

\bibitem{Halanay} {\sc A. Halanay}, {\em Differential equations: Stability, oscillations, time lags}, Academic Press, New York and London, 1966.

\bibitem{HaLee}
{\sc S.Y. Ha, K. Lee and D. Levy},
{\em Emergence of time-asymptotic flocking in a stochastic Cucker-Smale system},
{Commun. Math. Sci.,} 7, (2009), 453--469.


\bibitem{Ha-Liu}
{\sc S.-Y. Ha and J.-G. Liu},
{\em A simple proof of the Cucker-Smale flocking dynamics and mean-field limit},
{Commun. Math. Sci.,} 7, (2009), 297--325.


\bibitem{HaSlemrod}
{\sc S.Y. Ha and M.A. Slemrod},
{\em Flocking Dynamics of Singularly Perturbed Oscillator Chain and the Cucker-Smale System},
{ J. Dyn. Diff. Equat.}, 22, (2010), 325--330.

\bibitem{Tadmor-Ha}
{\sc S.-Y. Ha and E. Tadmor},
{\em From particle to kinetic and hydrodynamic descriptions of flocking},
{Kinet. Relat. Models}, 1, (2008), 315--335.

\bibitem{Lemercier}
{\sc S. Lemercier, A. Jelic, R. Kulpa, J. Hua, J. Fehrenbach,
P. Degond, C. Appert
Rolland, S. Donikian and J. Pettr\'{e}},
{\em Realistic following behaviors for crowd simulation},
{Comput. Graph. Forum}, 31, (2012), 489--498.


\bibitem{Mech}
{\sc N. Mecholsky, E. Ott and T. M. Antonsen},
{\em Obstacle and predator avoidance in a model for flocking},
{Physica D}, 239, (2010), 988--996.

\bibitem{Motsch-Tadmor}
{\sc S. Motsch, and E. Tadmor},
{\em A new model for self-organized dynamics and its flocking behaviour},
{J. Stat. Phys.}, 144, (2011), 923--947.

\bibitem{Pes}
{\sc J. Peszek},
{\em Existence of piecewise weak solutions of a discrete Cucker-Smale's flocking model with a singular communication weight},
{J. Differential Equations}, 257, (2014), 2900--2925.

\bibitem{PRT}
{\sc B. Piccoli, F. Rossi and E. Tr\'{e}lat},
{\em Control to flocking of the kinetic Cucker-Smale model},
{SIAM J. Math. Anal.,}  47, (2015), 4685--4719.

\bibitem{PR}  {\sc C. Pignotti and I. Reche Vallejo},
{\em  Flocking estimates for the Cucker-Smale model with time lag and hierarchical leadership},
{J. Math. Anal. Appl.,} 464, (2018), 1313-1332.

 \bibitem{PRV} {\sc C. Pignotti and I. Reche Vallejo}, {\em Asymptotic analysis of a Cucker-Smale system with leadership and distributed delay},
Trends in Control Theory and Partial Differential Equations, Springer Indam Ser., 32, (2019), 233--253.


\bibitem{PT} {\sc C. Pignotti and E. Tr\'elat},
{\em Convergence to consensus of the general finite-dimensional Cucker-Smale model with time-varying delays},
 Commun. Math. Sci., 16, (2018), 2053--2076.

\bibitem{Shen}
{\sc J. Shen},
{\em Cucker-Smale flocking under hierarchical leadership},
{SIAM J. Appl. Math.}, 68, (2007/08), 694--719.

\bibitem{Tan}
{\sc C. Tan},
{\em A discontinuous Galerkin method on kinetic flocking models},
{Math. Models Methods Appl. Sci.}, 27, (2017), 1199--1221.



\bibitem{Toscani}
{\sc G. Toscani},
{\em Kinetic models of opinion formation},
{Commun. Math. Sci.,} 4, (2006), 481--496.


 \bibitem{Vicsek}
{\sc T. Vicsek, A. Czirok, E. Ben Jacob, I. Cohen and O. Shochet},
{\em Novel type of phase
transition in a system of self-driven particles},
{Phys. Rev. Lett.}, 75, (1995), 1226--1229.


\end{thebibliography}
\end{document}